\documentclass[11pt, a4paper]{article}
\usepackage[english]{babel}
\usepackage{amssymb,amsmath,amsthm,bm,graphicx,natbib}
\usepackage{authblk}

\newtheorem{theorem}{Theorem}
\newtheorem{lemma}{Lemma}

\theoremstyle{definition}

\newcommand{\PP}{\mathbb{P}}

\newcommand{\EE}{{\mathbb{E}}}

\newcommand{\lle}{\,\,{\lesssim}\,\,}

\newcommand{\Var}{\mathbb{V}\mathrm{ar}}
\newcommand{\eps}{\varepsilon}

\newcommand{\FF}{\mathcal{F}}

\newcommand{\RR}{\mathbb{R}}
\newcommand{\NN}{\mathbb{N}}

\renewcommand{\phi}{\varphi}

\newcommand{\given}{\mid}
\renewcommand{\l}{\lambda}
\renewcommand{\L}{\Lambda}
\newcommand{\e}{\varepsilon}

\begin{document}

\title{\Large\bfseries Rate-optimal  Bayesian intensity smoothing for inhomogeneous Poisson processes}
\author[1]{\bfseries Eduard Belitser}
\author[2]{\bfseries Paulo Serra}
\author[3]{\bfseries Harry van Zanten}
\affil[1]{Department of Mathematics, VU University Amsterdam}
\affil[2]{Institute for Mathematical Stochastics, University of G\"ottingen}
\affil[3]{Korteweg-de Vries Institute for Mathematics, University of Amsterdam}
\date{\today}

\maketitle

\begin{abstract}
We apply nonparametric Bayesian methods to study the problem of estimating the intensity function of an inhomogeneous Poisson process.
To motivate our results we start by analysing count data coming from a call centre which we model as a Poisson process.
This analysis is carried out using a certain spline prior.
This prior is based on B-spline expansions with free knots, adapted from well-established methods used in regression, for instance.
This particular prior is computationally feasible.
Theoretically, we derive a new general theorem on contraction rates for posteriors in the setting of intensity function estimation which can be applied not just to this spline prior but also to a large number of other commonly used priors.
Practical choices that have to be made in the construction of our concrete spline prior, such as choosing the priors on the number and the locations of the spline knots, are based on these theoretical findings.
The results assert that when properly constructed, our approach yields a rate-optimal procedure that automatically adapts to the regularity of the unknown intensity function.
\end{abstract}


\bigskip

\section{Introduction}

Poisson processes have a long-standing history and are some of the most widely used processes in statistics to study temporal and 
spacial count data, in diverse fields such as communication, meteorology, seismology, hydrology, astronomy, biology, medicine, actuary sciences and queueing, among others.
In this paper we focus on inhomogenous Poisson processes on the real line with periodic intensity functions, which are models for count data in settings with a natural periodicity.
We obtain asymptotic results as the number of observed periods goes to infinity but our approach is flexible enough to also deliver asymptotic results for estimating an intensity function on a compact in terms the either the number of observed events or in terms of the scale of the intensity function.

Nonparametric Bayesian methods, which are used more and more in many different statistical settings, have so far 
only been used on a limited scale to analyze such models. From the applied perspective they can be attractive
for making inference about intensity functions, for the same reasons they are appealing in other situations. 
Estimating the intensity essentially requires some form of smoothing of the count data, and a nonparametric 
Bayesian approach can provide a natural way of achieving this. Using hierarchical priors we can automatically achieve
a data-driven selection of the degree of smoothing. 
Moreover, Bayesian methods provide a way to quantify the uncertainty about the intensity using the spread of the 
posterior distribution. A typical implementation provides a computational algorithm that can generate a large number 
of (approximate) draws from the posterior. From this it is usually straightforward to construct numerical credible bands or credible sets. 

The relatively small number of papers  using nonparametric Bayesian methodology for intensity function smoothing
have explored various possible prior distributions on intensities.
An early reference is  \cite{Moller}, who consider log-Gaussian priors. 
Other papers employing Gaussian process priors, combined with  suitable link functions, include
\cite{Adams} and  \cite{Palacios}. \cite{Kottas} consider kernel mixtures priors; see also the related paper \cite{DiMatteo}, 
in which count data is analysed using spline-based priors. 

The cited papers show that  nonparametric Bayesian inference for inhomogenous Poisson processes 
can give satisfactory results in various applications. On the theoretical side however the
existing literature provides no 
performance guarantees in the form of consistency theorems or related results. 
It is by now well known that nonparametric Bayes methods may suffer from inconsistency, even 
when seemingly reasonable priors are used (e.g.\ \citealt{Diaconis:1986}). The purpose of this paper is therefore to propose 
a Bayesian approach to nonparametric intensity smoothing that is both 
computationally feasible and at the same time theoretically underpinned by results on consistency and related 
issues like convergence rates and adaptation to smoothness. Such theoretical results have 
in the last decade been obtained for various statistical settings, including density estimation, 
regression, classification, drift estimation for diffusions, etcetera (see e.g.\ \citealt{Sub} for an overview
of some of these results). 
Until now,  intensity estimation for inhomogenous Poisson processes  has remained largely unexplored.

As motivation and starting point for the paper we consider the problem of analysing count data from a  call center. 
The same type of data were analyzed by frequentist methods in the paper \cite{Brown:2005}. 
We revisit the problem using a  nonparametric Bayesian method employing a spline-based prior on the unknown intensity function. 
In addition to a single estimator of the intensity, this method provides credible bounds indicating the degree of uncertainty.
In Section \ref{sec:theory} we study theoretical properties of our procedure, namely consistency, posterior contraction rates and adaptation to smoothness. 
The results show that we have set up our procedure in such a way that we obtain consistent, rate-optimal estimation of the intensity and that the method adapts automatically to the unknown smoothness of the intensity curve, up to the level of the order of the splines that are used. 
Section \ref{sec:con} concludes with some remarks and directions for further research.

\section{Analysis of call center data}
\label{sec:count}

\subsection{Data and statistical model}
\label{sec:data}

The approach we propose and study is motivated by the wish to analyse 
 a dataset consisting of counts of telephone calls arriving 
at a certain call center. 
The dataset was obtained from the website of the S.E.E.\ Center (http://ie.technion.ac.il/Labs/Serveng/) of the Faculty of Industrial Engineering and Management, Technion in Haifa, Israel.
It consists of counts for calls arriving at a bank's 24 hour a day call center in the United States of America.
We considered the records for the period from November 1,  2001 until December 31,  2001, 
covering a total of about 2.8 million incoming phone calls. 
These events are recorded in 30 second intervals with an average of about 32 calls per minute.
The raw data are plotted in Figure \ref{fig1}.

\begin{figure}[h]
\includegraphics[width=14cm, height=6.56cm]{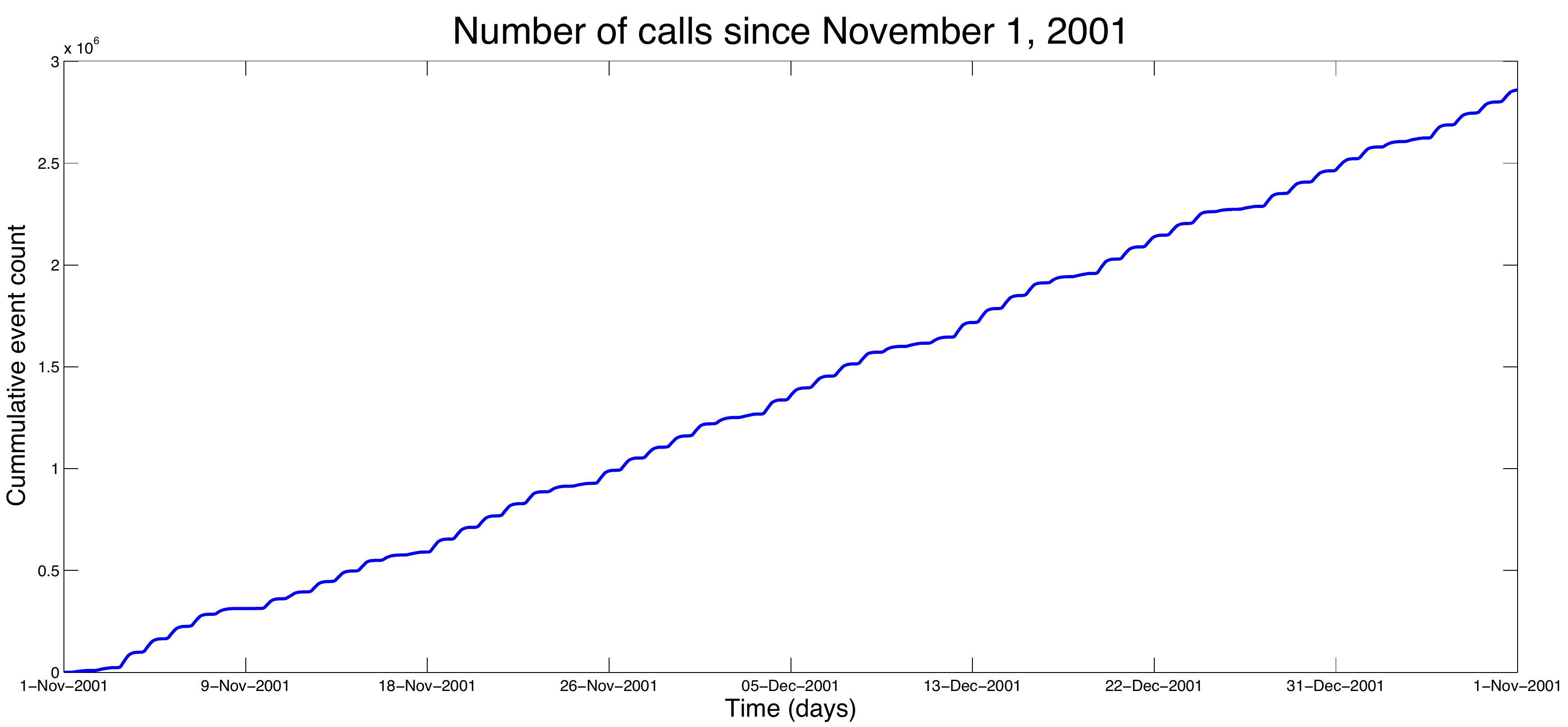}
\caption{Number of incoming phone calls between November 1, 2001 and December 31, 2001.}
\label{fig1}
\end{figure}

We model the full count data as the realization of an inhomogenous Poisson process $N$ with an intensity function $\lambda$ that 
is periodic, the period being $24$ hours (\citealt{Daley:1988}). 
This Poisson assumption is natural and is investigated in some detail in~\cite{Brown:2005}, who could not find significant evidence 
to the contrary in a similar dataset (same kind of data, but over a different time interval).
See also~\cite{Belitser:2012a}, who study the periodicity in the data.

This dataset is known to exhibit periodicity on different time scales; (approximate) daily, weekly, monthly and yearly periodicities seem to be present in the data.
Different time scales are relevant if one would like to take analyze the intensity of the call traffic during, say, the weekends, holidays or specific times of the year.
(To analyze the intensity of calls during weekends, for example, a weekly time scale would be appropriate.)
By carrying out our estimation procedure under the assumption of daily periodicity we are in fact estimating the average daily call intensity between November 1, 2001 and December 31, 2001.
Our study of the data over 24 hour intervals (the smallest time interval over which there is evidence of periodicity; cf. Figure~\ref{fig:random}) is motivated by the fact that the volume of calls in the dataset is already quite high even over individual days.

Let $n$ be the number of days for which we have data ($n = 61$) and let $T$ be the period ($24$ hours). 
Then the full call arrival counting process is given by 
$N = (N_t: t \in [0,nT])$, where $N_t$ is the number of calls arriving in the time interval  $[0,t]$. 
The Poisson assumption means that for every $0 \le s \le t$, the number of arrivals $N_t - N_s$ is independent 
of the history $(N_u: u \le s)$ up till time $s$ and that is has a Poisson distribution with mean 
\[
\int_s^t \lambda(u)\,du. 
\]
We will assume throughout that $\lambda$ is at least a continuous function. The periodicity assumption then means that
 $\lambda$ is a $T$-periodic function, i.e.\ $\lambda(t + T) = \lambda(t)$ 
for all $t \ge 0$. 
For $i = 1, \ldots, n$ we define the counting process $N^{(i)} = (N^{(i)}_t: t \in [0,T])$ by 
\[
N^{(i)}_t = N_{(i-1)T + t} - N_{(i-1)T}, \qquad t \in [0,T], 
\]
i.e.\ $N^{(i)}$ counts the number of arrivals during day $i$. Note that by the independence of the increments of the process $N$, 
the processes $N^{(i)}$ are independent inhomogenous Poisson processes which have the
restriction of $\lambda$ to $[0,T]$ as intensity function. 

Our goal is to make inference about this function.
Note that we do not observe the full process $N$. We only observe it at discrete times, namely every $30$ seconds. 
On average about $16$ calls arrive in a $30$ second time interval, so we really only see aggregated counts. 
Let $\Delta$ be the time between observations ($30$ seconds) and let $m = T/\Delta$ be the number of counts per day 
that we have in our dataset ($m =2880$ in our case). Then for every $i =1, \ldots, n$ and $j = 1, \ldots, m$, the number of arrivals
\begin{equation}\label{eq:c}
C_{ij} = N^{(i)}_{j\Delta} - N^{(i)}_{(j-1)\Delta} 
\end{equation}
in the $j$th time interval on day $i$ has a Poisson distribution with parameter
\begin{equation}\label{eq:l}
\lambda_j = \int_{(j-1)\Delta}^{j\Delta} \lambda(t)\,dt. 
\end{equation}
We denote the total available count data by $C^n = (C_{ij}: i=1, \ldots, n, j=1, \ldots, m)$. 
It follows that the likelihood is given by 
\begin{equation}\label{eq:lik1}
p(C^n \given \lambda) = \prod_{i=1}^n\prod_{j=1}^m \frac{\lambda_j^{C_{ij}}e^{-\lambda_j}}{C_{ij}!}.
\end{equation}

The case of discretized data is more relevant from the practical point of view.
In the following section we describe a prior one may place on the intensity function $\lambda$.

\subsection{Prior on the intensity function}
\label{sec:prior}

There are different possible choices of priors for the function $\lambda$.
A number of options considered earlier in the literature were already mentioned in the introduction (Gaussian processes, 
kernel mixtures, splines). The particular prior we apply in this paper to illustrate our results is motivated 
by the desire to have a computationally manageable procedure on the one hand and theoretical 
performance guarantees on the other. Still there will conceivably be more than one 
sensible choice meeting these requirements. In this section we restrict our efforts to the 
investigation of a specific spline-based prior which is described in detail in this section.
We would like to clarify that we use this particular prior due to its computational simplicity and to illustrate our results from Section~\ref{sec:theory}.
Our theoretical results are more general and in fact cover this spline prior as a particular case.

More precisely we will employ a certain free-knot spline prior which is  similar 
to priors considered earlier in different contexts (see for instance \citealt{SmithKohn96, Denison, DiMatteo}, or, more recently \citealt{Sharef} 
and the references therein). Such priors  have proven to be numerically attractive 
and capable of capturing abrupt changes in functions of interest. This last 
point is relevant for our particular application, since we expect fluctuations 
during the day due to the varying activity of businesses over the day.
Recently, several theoretical results were derived 
for spline-based priors in various setting as well (e.g.\ \citealt{Ghosal:2000,Ghosal:2008, deJonge:2012, Belitser:2012b}). We will show in the next section 
that the procedure that we construct and implement has several desirable theoretical properties.

Background information on splines can be found, for example, 
in \cite{deBoor:2001} or \cite{Schumaker:2007}. 
Let us fix some notations and terminology. 
A function is called a spline of order $q \in \NN$, with respect to a certain partition 
of its support, if it is $q-2$ times continuously differentiable and when restricted 
to each interval in the partition, it coincides with a polynomial of degree at most $q-1$. 
Now consider $q \ge 2$. 
For any $j\in\mathbb{N}$, such that $j\ge q$ let $\mathcal{K}_j = 
\{( k_1, \dots, k_{j-q} )\in (0, T)^{j-q}: 0<k_1< \dots < k_{j-q}<1 \}$.
We will refer to a vector $\bm{k}\in\mathcal{K}_j$ as a sequence of {\em inner knots}.

A vector $\bm{k}\in\mathcal{K}_j$  induces the partition 
$\big\{ [k_0, k_1), [k_1, k_2), \dots, [k_{j-q}, k_{j-q+1}]\big\}$ of $[0,T]$, 
with $k_0 = 0$ and $k_{j-q+1} = T$. 
For  $\bm{k} \in\mathcal{K}_j$, we denote by $\mathcal{S}^{\bm{k}}=\mathcal{S}^{\bm{k}}_q$ 
the linear space of splines  of order $q$ on $[0, T]$ with simple 
knots $\bm{k}$ (see the definition of simple knots  in, e.g.,  \cite{Schumaker:2007}).
This space has dimension $j$ and admits a basis of B-splines 
$\{B_1^{\bm{k}},\ldots, B_j^{\bm{k}}\}$. 
The construction of $\{B_1^{\bm{k}},\ldots, B_j^{\bm{k}}\}$  involves  
the knots $k_{-q+1},\ldots, k_{-1}, k_0, k_1,$ $\ldots,$ $k_{j-q}, k_{j-q+1}, k_{j-q+2},$ $\ldots$ $,k_j$, 
with arbitrary extra knots $k_{-q+1} \le \cdots \le k_{-1} \le k_0 =0$ and 
$T=k_{j-q+1} \le k_{j-q+2} \le \cdots \le k_j$. 
Usually one takes $k_{-q+1}=\cdots =k_{-1}= k_0 =0$ and 
$T=k_{j-q+1}=\cdots =k_j$, and we adopt this choice as well.
For $\bm{k} \in\mathcal{K}_j$ and $\bm{\theta} \in \RR^j$ we denote by $s_{\bm{\theta}, \bm{k}}$ 
the spline in $\mathcal{S}^{\bm{k}}$ that has coefficient vector $\bm{\theta}$
relative to the basis $\{B_1^{\bm{k}},\ldots, B_j^{\bm{k}}\}$, i.e.\
\[
s_{\bm{\theta}, \bm{k}}(t) = \sum_{i = 1}^j \theta_i B_j^{\bm{k}}(t), \qquad t \in [0,T].
\]

To define our  prior $\Pi$ on $\lambda$ we first fix the order $q \ge 2$ of the splines that we use (cubic splines
are popular, they correspond to the choice $q= 4$) and the minimum and maximum intensities $0 \le  M_1 < M_2$. 
Then a draw from the prior $\Pi$ is constructed as follows:
\begin{enumerate}
\item(Number of B-splines): Draw $J \ge q$ according to a shifted Poisson distribution with mean $\mu$.
\item(Location of the knots): Given $J = j$,  construct a regular $1/j^2$-spaced grid in 
$(0,T)$. Then uniformly at random, choose $j-q$ grid elements  (without replacement)
to form a sequence of inner knots $\bm{k}$. 
\item(B-spline coefficients): Also given $J = j$, and independent of the previous step, 
draw a vector $\bm{\theta}$ of $j$ independent, uniform $U[M_1,M_2]$-distributed 
B-spline coefficients. 
\item(Random spline): Finally, construct the random spline $s_{\bm{\theta}, \bm{k}}$
of order $q$ corresponding to the inner knots $\bm{k}$ and with B-spline 
coefficient vector $\bm{\theta}$. 
\end{enumerate}

The specific choices made in the construction of the prior, like the Poisson distribution 
on $J$, choosing the knots uniformly at random from a grid, etcetera, are 
motivated by the optimality theory that we derive in Section \ref{sec:theory}. 
The theory shows that there is some more flexibility, but for 
choices too far from the ones proposed above the performance guarantees brake down. 
Technically, the prior on $\lambda$ is the measure $\Pi$ 
on the space $C[0,T]$ of continuous functions on $[0,T]$ given by the law, or 
distribution of the random spline $s_{\bm{\theta}, \bm{k}}$ described above. 
The splines in $\mathcal{S}^{\bm{k}}_q$ are $q-2$ times continuously differentiable, 
hence in this sense the choice of $q$ determines the regularity of the prior. 
We will see in the next section that it also determines the maximal degree of 
smoothness of the true underlying intensity to which our procedure can adapt. 
In applications like the one we are interested in here, a sensible choice of the parameters 
$M_1$ and $M_2$ will typically be suggested by the average number of counts per time unit in the data. 
The construction of the grid in step 2.\  is non-standard compared to other spline-based priors proposed in the literature.
It is motivated by recent work of \cite{Belitser:2012b} and will allow us to derive desirable theoretical properties in the next section.

\subsection{Posterior inference}
\label{sec:post}

For the data described in Section \ref{sec:data}, with  likelihood \eqref{eq:lik1}, and 
the spline prior $\Pi$ described in Section \ref{sec:prior}, 
we implemented an MCMC procedure to sample from the corresponding 
posterior distribution of the intensity function $\lambda$ of interest.  
The minimal and maximal intensity parameters $M_1$ and $M_2$ were set to 
$200$ and $20000$, respectively. These numbers were motivated by the range of the data (time 
is measured in hours). We took the order $q$ of the splines equal to 4.

Since our prior is very similar to the ones used previously in for instance 
\cite{DiMatteo} or \cite{Sharef} in regression or hazard rate estimation settings, 
our computational algorithm is a rather straightforward adaptation of existing methods. 
A generic state of the chain is a $(2J-q+1)$-dimensional vector $(j, \bm{k}, \bm{\theta})$ where $j\in\mathbb{N}$, $j\ge q$ is the model index, $\bm{k}=\bm{k}_j\in(0, T)^{j-q}$ is a vector of inner knots and $\bm{\theta}=\bm{\theta}_j\in\mathbb{R}^{j}$ is a vector of B-spline coordinates.
Together, these index a spline $s_{\bm{\theta},\bm{k}}=s_{\bm{\theta}_j,\bm{k}_j}^q\in S_{\bm{k}_j}^q$.
We will abbreviate the corresponding posterior distribution by $\pi(j, \bm{k}, \bm{\theta}|C^n)$.
Since the splines involved are easy to evaluate and integrate we can compute the likelihood, and then the posterior, up to the normalization constant, without any approximations being needed.

We consider four different types of moves for the MCMC chain, namely: a) perturbing the coefficients, b) moving the location of one knot, c) birth of a new knot and d) death of an existing knot.
Each of these moves is proposed, independently and respectively, with probabilities $p_a$, $p_b$, $p_c(j)$ and $p_d(j)$ where for each $j\ge q$, $p_a + p_b + p_c(j) + p_d(j) = 1$.
In fact, we start by picking $0 < p_a + p_b < 1$ as parameters of the algorithm;
if $\mu$ is the mean of the prior on $J$, then we take $p_c(q) = 1 - p_a - p_b$, $p_d(q) = 0$ and, for $j>q$, $p_c(j) = (1- p_a - p_b) 2^{-\frac{j-q}{\mu-q}}$ and $p_d(j) = (1- p_a - p_b) (1-2^{-\frac{j-q}{\mu-q}})$.
This choice results in $p_c(\mu) = p_d(\mu)$ if $j=\mu$, $p_c(\mu) > p_d(\mu)$ if $j<\mu$, $p_c(\mu) > p_d(\mu)$ if $j<\mu$.

When perturbing the coefficients we perform simple (Gaussian) random walk MCMC steps; the standard deviation of the random walk was chosen such that we obtained an acceptance rate  of roughly 23\% for this type of move, as prescribed in 
\cite{Roberts:1997}.
Let $\phi_j$ be the joint density of $j$ i.i.d.\ standard normal random variables.
Our proposals correspond to a move $(j, \bm{k}, \bm{\theta})\to(j, \bm{k}, \bm{\theta} + \sigma \bm{u})$ which we accept with probability $\min\big(A(\bm u),1\big)$, with
\begin{equation*}
A(\bm u)
=	\frac{ \pi(j, \bm{k}, \bm{\theta} + \sigma \bm{u} | C^n)\, p_a\, \varphi_{j}(-\sigma \bm{u}) }{ \pi(j, \bm{k}, \bm{\theta}| C^n)\, p_a\, \varphi_{j}(\sigma \bm{u}) }
=	\frac{ \pi(j, \bm{k}, \bm{\theta} + \sigma \bm{u}| C^n) }{ \pi(j, \bm{k}, \bm{\theta}\,\,| C^n) }.
\end{equation*}

Moving a knot is also straightforward; one of the current $j-q$ knots, say $k_i$, is picked uniformly at random among those in $\bm{k},$ and we propose to change its location depending on how many of its neighboring position on the $j^{-2}$-spaced grid are free -- we say that two knots $k, k'$ are neighbors if $|k-k'|\le j^{-2}$.
This means that we propose a move $(j, \bm{k}, \bm{\theta})\to(j, \bm{k}', \bm{\theta})$ where $\bm{k}$ and $\bm{k}'$ differ only at the $i$-th position:
if $k_i$ has two free neighboring positions, then it moves to either of them with equal probability $c_i=c_i(k_{i-1},k_i,k_{i+1})=1/2$;
if $k_i$ only has one free neighboring position, then, with equal probability $c_i=1/2$, it either moves to this free position or it does not move at all;
if $k_i$ has no free neighboring positions then if does not move, with probability $c_i=1$.
These particular choices assure the reversibility of the moves.
We accept such a proposal with probability $\min\big(A(i),1\big)$ where $A(i)$ is given by
\begin{equation*}
A(i)
=	\frac{ \pi(j, (k_1, \dots,k_i',\dots,k_{j-q}), \bm{\theta}| C^n)\, p_b\, (j-q)^{-1}\, c_i }{ \pi(j,(k_1, \dots,k_i,\dots,k_{j-q}) , \bm{\theta}| C^n)\, p_b\, (j-q)^{-1}\, c_i }
=	\frac{ \pi(j, \bm{k'}, \bm{\theta}| C^n) }{ \pi(j, \bm{k}\,\, , \bm{\theta}| C^n) }.
\end{equation*}

Birth moves and death moves, where a new knot is respectively added and removed, are reverse moves of one another and so we will outline only how to perform the birth move.
We propose a move $(j, \bm{k}, \bm{\theta})\to(j+1, \bm{k}', \bm{\theta}')$ where we add a new knot to the vector $\bm{k}$ and a new coefficient to the vector $\bm{\theta}$.
In doing so, a new B-spline is introduced to the B-spline basis and a new B-spline coefficient is generated.
The new knot vector $\bm{k}'$ contains all knots from $\bm{k}$ rounded to the closest grid point on a $(j+1)^{-2}$ spaced grid with the extra knot then picked uniformly at random among the remaining free positions; call it $k_{i-1} < k' < k_i$.
Note that this construction does not prevent two knots in $\bm{k}'$ from occupying the same position;
such knot vectors have posterior probability 0, though, so that the probability of moving to such a state is zero.
The coefficients on this basis are then picked as $\bm{\theta}'=f(\bm{\theta},u) = (\theta_1, \dots, \theta_{m-1}, u, \theta_m, \dots, \theta_{j})$ where $f$ will be linear and invertible, and $u$ is a {random seed}, a normally distributed random number with mean $\eta(\bm{\theta})$, to be picked later, and variance 1.
The new knot will belong to the support of $q$ B-splines, namely the $i$-th through $(i+q-1)$-th B-splines and we pick the index $m$ in $\{i, \dots, i+q\}$ depending on the knot's position within the interval $[k_{i-1}, k_i]$; namely $m = i + \lfloor(q+1)(k'-k_{i-1})/(k_i-k_{i-1})\rfloor$, where $\lfloor a \rfloor$ is the largest integer smaller or equal to $a$.
The mean of the random seed $u$ will be picked as a weighted mean of the coefficients $\bm{\theta}$, namely, $\eta(\bm{\theta})=\sum_{i=1}^{m-1} w_i \theta_i+\sum_{i=m}^j w_{i-1} \theta_i$, where the weights $w_i$ are normalized and
\[
w_i \propto\; \int_0^TB_{\bm{k}_j',m}^q(t) B_{\bm{k}_j',i}^q(t)\,dt,\quad i=1, \dots, j+1.
\]
With probability $\min\big(A(j, k', u),1\big)$ we make the move $(j, \bm{k}, \bm{\theta})\to(j+1, \bm{k}', \bm{\theta}')$, with $\bm{\theta}'=f(\bm{\theta}, u)$ and $\bm{k}' = (k_1, \dots, k_{i-1}, k', k_i, \dots, k_{j-q})$, where
\begin{eqnarray*}
A(j, k', u) = \frac{ \pi(j+1, \bm{k}', f(\bm{\theta},u)| C^n)\, (1-p_a-p_b)p_d(j+1)\, (j-q+1)^{-1} }
{ \pi(j, \bm{k}, \bm{\theta}| C^n)\, (1-p_a-p_b)p_c(j)\, (j^{2}-j+q)^{-1} \varphi_1(u)  }|J_f|
\end{eqnarray*}
where $|J_f|$ is the Jacobian of the linear mapping described before.

Figure~\ref{fig:random} summarizes the outcome of the analysis. In the top panel it shows 
the posterior and 95\% point-wise credible intervals, based on 10,000 samples from 
posterior. 
The lower panel shows a histogram for the locations of the knots corresponding to the 
samples from the chain used to generate the top panel. Note that as expected,  
relatively many knots are placed in periods in which there are relatively many fluctuations in the intensity.

\begin{figure}[h!]
\includegraphics[width=14cm, height=6.56cm]{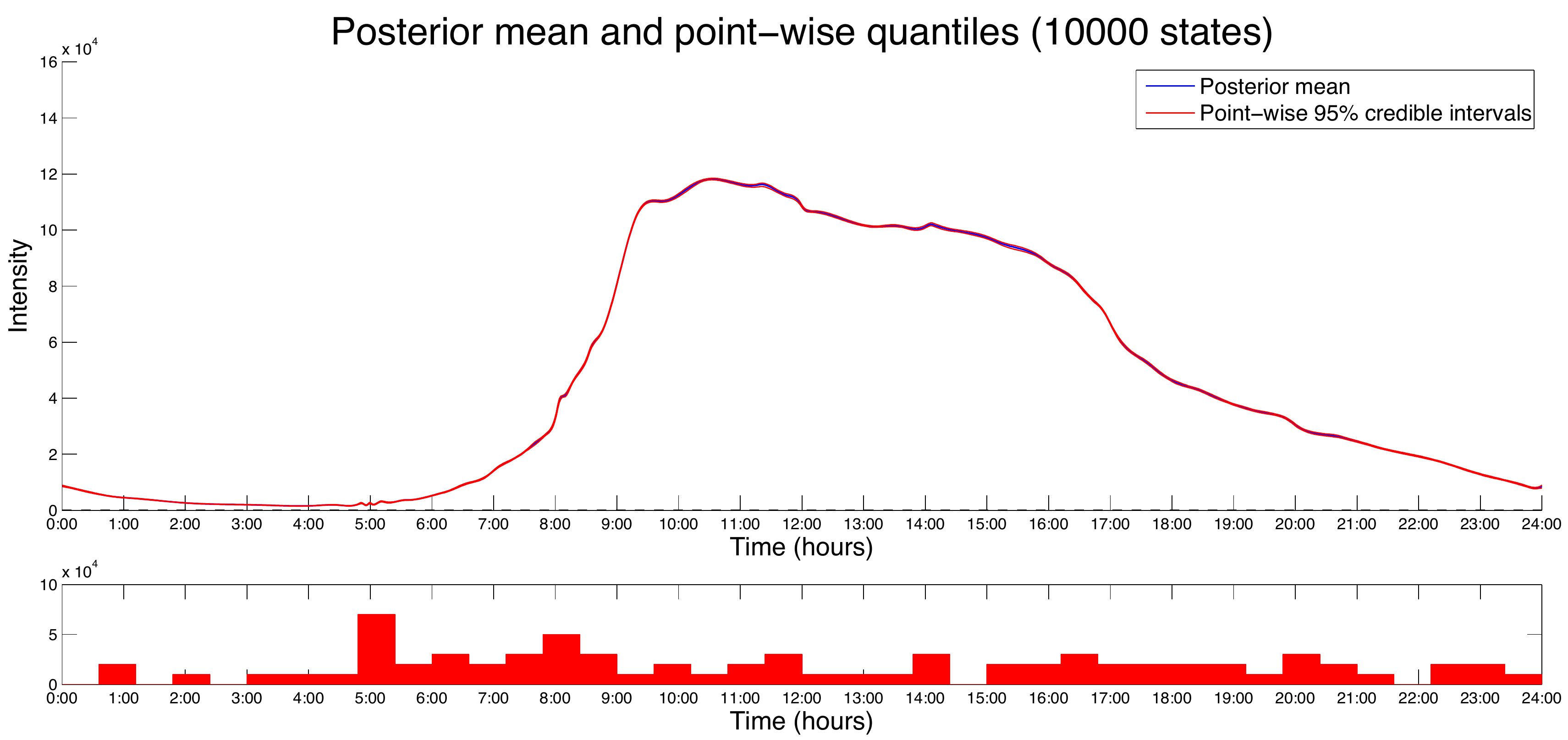}
\caption{
Top panel: posterior distribution of the intensity function $\lambda$ based on the thinned data. (Blue: posterior mean, red: point-wise $95\%$ credible intervals).
Lower panel: posterior distribution of the knot locations (Histogram).}
\label{fig:random}
\end{figure}

Due to the large event rate (almost $3$ million counts in total), the credible bands are very narrow.
To illustrate the dependence on the amount of data we ran the analysis again with a thinned out 
dataset. We randomly removed counts, retaining about $1,000$ counts. The same analysis 
then leads to the posterior plot given in Figure \ref{thinned}. In this case, the 
uncertainty in the posterior distribution becomes clearly visible.

\begin{figure}[h!]
\begin{center}
\includegraphics[width=14cm, height=6.56cm]{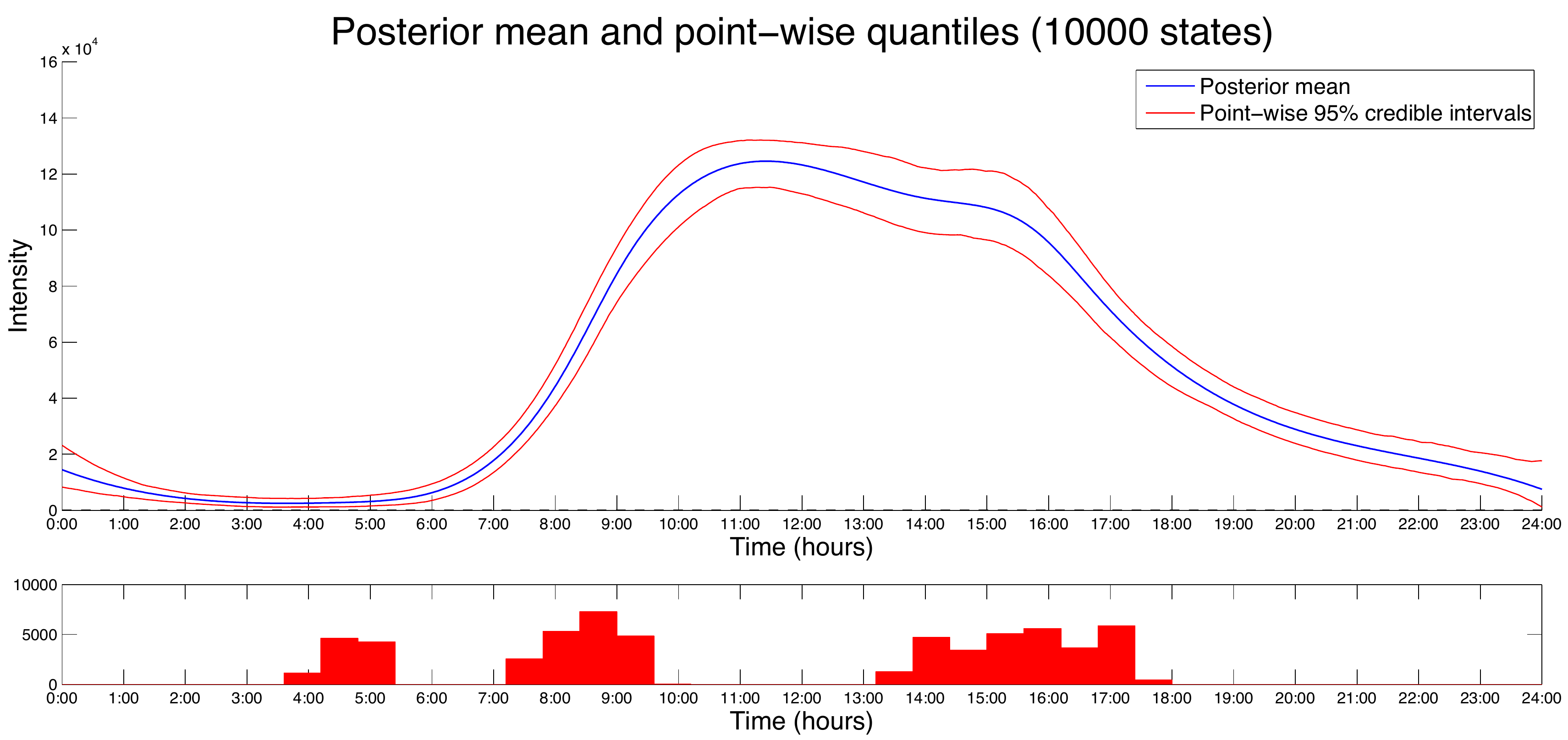}
\end{center}
\caption{
Top panel: posterior distribution of the intensity function $\lambda$ based on the thinned data. (Blue: posterior mean, red: point-wise $95\%$ credible intervals).
Lower panel: posterior distribution of the knot locations (Histogram).}
\label{thinned}
\end{figure}

We find that  the prior that we defined in Section \ref{sec:prior} is a computationally feasible 
choice for nonparametric Bayesian  intensity smoothing in the context of this kind of periodic count data. 
In the next section we analyze its fundamental theoretical performance. 
See in particular Theorem \ref{theorem:spline} in Section \ref{sec:spline}.

\section{Theoretical results}
\label{sec:theory}

\subsection{Contraction rates for general priors}

We derive our theoretical results for the particular prior we 
used in the Section \ref{sec:count} from general rate of contraction results that we present in this section. 
These are in the spirit of the general theorems about convergence rates of nonparametric Bayes 
procedures that were first developed for density estimation (\citealt{Ghosal:2000}) 
and later for various other statistical settings; see for instance \cite{vanderMeulen:2006}, 
\cite{Ghosal:2007}, \cite{Panzar:2009}. 
Here we complement this literature with general rate results regarding intensity estimation for 
inhomogenous Poisson processes. 
These results are not only applicable to the spline priors we consider in this paper, 
but may also be used to analyze contraction rates of other priors. 
Moreover, we formulate the theorems not just for the case that we have discrete observations of 
aggregated data, as in our data example, but also for the case that the full counting process is observed. 

The setting is as in Section \ref{sec:data}. We fix a period $T > 0$. In the full observations case 
we assume that for $n \in \NN$, we observe an inhomogeneous Poisson process $N^n= (N^n_t: t \in [0,nT])$ up till time $nT$, 
with a $T$-periodic intensity function $\lambda$.
Equivalently, we can say we observe $n$ independent inhomogeneous Poisson processes $N^{(1)}, \ldots, N^{(n)}$, 
indexed by $[0,T]$, and with a common  intensity function $\lambda$, which is a positive, 
integrable function on $[0,T]$. It is well known that the law of $N$ under the intensity function $\lambda$
is equivalent to the law of a standard Poisson process and that the corresponding likelihood is given by 
\begin{align}\label{ref:lik}
p(N^n \given \lambda) = e^{-\int_0^{nT}(\lambda(t)-1)\,dt + \int_0^{nT} \log(\lambda(t))\,dN^n_t}
\end{align}
(see for instance \citealt{Jacod:2003}).

In the remainder of this section we will derive results will involve asymptotics in the number of observed periods $n$.
Alternatively, one could group the data and define a new Poisson process $M^n=(M^n_t: t \in [0,T])$ on $[0, T]$,
\[
M_t^n = \sum_{i=1}^n N_t^{(i)}, \quad t\in[0,T],
\]
which is a Poisson process with intensity $n\lambda$ on $[0,T]$.
By identifying the function $\lambda(t)$ on $[0,T]$ with its periodical extension $\lambda(t\pmod{T})$ on $\mathbb{R}$, the likelihood of this process equals \eqref{ref:lik}.
If we assume that $\lambda$ is (upper) bounded, the results from this section also imply asymptotics in terms of the number of events (or scale) of the intensity function $\lambda$.
Asymptotics in the length of the trajectory of the process that is observed is a third equivalent formulation for our results.

We consider prior distributions that charge strictly positive, continuous functions. 
Given such a prior $\Pi_n$ on $\lambda$ (which we allow to depend on $n$) 
we can then compute the  corresponding posterior distribution $\Pi_n(\cdot \given N^n)$ by Bayes' 
formula
\[
\Pi_n(\lambda \in B\given N^n) = \frac{\int_B p(N^n\given \lambda) \,\Pi_n(d\lambda)}
{\int p(N^n\given \lambda) \,\Pi_n(d\lambda)}.
\]
Formally we can view the prior  $\Pi_n$ as a measure on the 
space $\Lambda \subset C[0,T]$ of 
all continuous, strictly positive functions on $[0,T]$, endowed with its Borel $\sigma$-field.  If we endow 
$\Lambda$ with the uniform norm, the likelihood is a continuous function on $\Lambda$.
Hence, the posterior is a well-defined measure on $\Lambda$.

The following theorem considers the frequentist setting in which the data are assumed to be 
generated by an unknown, ``true'' intensity function $\lambda_0$. 
It gives conditions on the prior $\Pi_n$ under which 
the posterior $\Pi_n( \cdot\given N^n)$
contracts around the true $\lambda_0$ at a certain rate as the number of observed periods tends to infinity. 
The assumptions and conclusions of the theorem are formulated in terms of 
various distances on the intensity functions. For a continuous function $f$ on $[0,T]$ we define
the norms $\|f\|_2$ and $\|f\|_\infty$ as usual by 
\[
\|f\|^2_2 = \int_0^T f^2(t)\,dt, \qquad \|f\|_\infty = \sup_{t \in [0,T]}|f(t)|. 
\]
For a set of positive continuous functions $\FF$ we write $\FF^c$ for its complement and $\sqrt\FF = \{\sqrt f: f \in \FF\}$.
For $\eps > 0$ and a norm $\|\cdot\|$ on $\FF$, let $N(\eps, \FF, \|\cdot\|)$ be the minimal number of balls of $\|\cdot\|$-radius $\eps$ needed to cover $\FF$.

\begin{theorem}[Contraction rate for full observations]\label{theorem:con}
Assume that $\l_0$ is bounded away from $0$.
Suppose that for positive sequences $\tilde\eps_n, \bar\e_n \to 0$ such that $n(\tilde \e_n \wedge\bar\eps_n)^2 \to \infty$ as $n\to\infty$, and  constants $c_1, c_2 > 0$ it holds that for all $c_3 > 1$, there exist subsets $\Lambda_n \subset \Lambda$ and a constant $c_4 > 0$ such that
\begin{align}
\label{eq:pm}\Pi_n(\l: \|\l-\l_0\|_\infty \le \tilde\e_n) & \ge c_1e^{-c_2n\tilde\eps^2_n}, \\
\label{eq:rm}\Pi_n(\Lambda_n^c) & \le e^{-c_3 n\tilde\eps^2_n},\\
\label{eq:en}\log N(\bar\eps_n, \sqrt\Lambda_n, \|\cdot\|_2) & \le c_4n\bar\eps^2_n.
\end{align}
Then for $\e_n = \tilde \e_n \vee \bar\e_n$ and  all sufficiently large $M > 0$, 
\begin{equation}\label{eq:ster}
\EE_{\l_0}\Pi_n(\lambda \in \Lambda: \|\sqrt\l-  \sqrt\l_0\|_2 \ge M\e_n \given N^n ) \to 0
\end{equation}
as $n \to \infty$. 
\end{theorem}

The proof of this theorem is given in  Appendix \ref{sec:a1}. 
The assumptions of the theorem parallel those of similar theorems obtained
earlier for other settings including density estimation, regression, and 
classification. The first condition \eqref{eq:pm}, the 
{\em prior mass condition}, requires that the prior
puts sufficient mass near the truth. Conditions \eqref{eq:rm}--\eqref{eq:en}
together require that most of the prior mass, quantified in the sense 
of the {\em remaining mass condition} \eqref{eq:rm}, is concentrated 
on {\em sieves} $\L_n$ which are ``small'' in the sense of metric entropy, 
quantified by the {\em entropy condition} \eqref{eq:en}.

The condition requiring that $\lambda_0$ be bounded away from zero, while needed in the proof, is essentially innocuous.
Indeed, if this assumption does not hold (or if it is not know whether it holds), one might simply modify the Poisson data by adding to it an independently generated homogeneous Poisson process with intensity $1$, say.
The resulting data can be seen as a realization of a Poisson process with intensity $1+\lambda_0$ (which is bounded away from zero by at least $1$) and Theorem~\ref{theorem:con} can be applied to it to make inference on $\lambda_0+1$ and therefore on $\lambda_0$.
This effectively allows us to make inference on intensities which are not bounded away from zero.

The proof of the theorem shows that conditions \eqref{eq:rm}--\eqref{eq:en}
can in fact be slightly weakened, at the cost of using more complicated distance
measures on the intensities. The conditions in the theorem are more 
easy to work with when studying concrete priors and are expected to 
give sharp results in many cases.
We note that if under the prior all intensities are bounded away from $0$, 
then the set $\sqrt{\Lambda_n}$ in \eqref{eq:en} may be replaced by 
$\Lambda_n$. Moreover, if all intensities are uniformly bounded by 
a common constant under the prior, then the square-root norm $\|\sqrt\cdot\|_2$
in \eqref{eq:ster} may be replaced by the $L^2$-norm $\|\cdot\|_2$ itself. 
In the next section we verify the conditions of the theorem for the spline priors 
used in Section \ref{sec:data}.

In the case of discrete observations we only have access, for some $m \in \NN$ and $\Delta = T/m$, to 
aggregated counts $C_{ij}$ for $i = 1, \ldots, n$ and $j =1, \ldots, m$, given by \eqref{eq:c}.
As before, we summarize these data using the notation $C^n = (C_{ij}: i=1, \ldots, n, j=1, \ldots, m)$. 
As explained in Section \ref{sec:data} the likelihood is in that case given by \eqref{eq:lik1}, 
where the $\lambda_j$'s are defined as in \eqref{eq:l}. Consequently, the 
discrete-observations posterior is given by 
\[
\Pi_n(\lambda \in B\given C^n) = \frac{\int_B p(C^n\given \lambda) \,\Pi_n(d\lambda)}
{\int p(C^n\given \lambda) \,\Pi_n(d\lambda)}.
\]
In this case it is clear that we can not consistently identify the whole intensity function $\l$
from the data, but only the integrals $\lambda_1, \ldots, \lambda_m$. 
In the following theorem, which deals with the convergence of the posterior distribution
in the case of discrete observations, we therefore measure the convergence using a semi-metric 
that identifies intensity functions with the same integrals over time intervals 
in which we make observations.
For $\l, \l' \in \L$, we define the distance $\rho$ by setting 
\[
\rho^2(\lambda, \l')  = \sum_{j=1}^m \Big(\sqrt{\l_j\phantom{'}} - \sqrt{\l_j'}\Big)^2
= \sum_{j=1}^m \Big(\sqrt{\int_{(j-1)\Delta}^{j\Delta} \lambda(t)\,dt} - 
\sqrt{\int_{(j-1)\Delta}^{j\Delta} \lambda'(t)\,dt}\Big)^2.
\]
The theorem has exactly the same assumptions on the prior as Theorem \ref{theorem:con} above, but 
gives a contraction rate relative to the distance $\rho$.

\begin{theorem}[Contraction rate for discrete observations]\label{theorem:dis}
Assume that $\l_0$ is bounded away from $0$. 
Suppose that for postive sequences $\tilde\eps_n, \bar\e_n \to 0$ such that $n(\tilde \e_n \wedge\bar\eps_n)^2 \to \infty$ as $n\to\infty$,
and  constants $c_1, c_2 > 0$  it holds that for all $c_3 > 1$, there exist 
  subsets $\Lambda_n \subset \Lambda$ and a constant $c_4 > 0$ such that  
\eqref{eq:pm}--\eqref{eq:en} hold. 
Then for $\e_n = \tilde \e_n \vee \bar\e_n$ and  all sufficiently large $M > 0$, 
\[
\EE_{\l_0}\Pi_n(\lambda \in \Lambda: \rho(\l, \l_0) \ge M\e_n \given C^n ) \to 0
\]
as $n \to \infty$. 
\end{theorem}
The proof of the theorem is given in Appendix \ref{sec:a2}

In the next section we apply the theoretical results derived above to the spline prior considered before.

\subsection{Contraction rates for the spline prior}
\label{sec:spline}

Having the general rate of contraction results given by Theorems \ref{theorem:con} 
and \ref{theorem:dis} at our disposal we can use them to study the performance
of the spline-based prior defined in Section \ref{sec:prior}. We fix the order $q \ge 2$ 
of the splines that are used. 
As before, let $N^n$ be the a full path up till time $nT$ of an inhomogenous Poisson process
$N$ with $T$-period intensity $\l_0$ and let $C^n$ be the discrete-time counts 
$C^n = (C_{ij}: i=1, \ldots, n, j=1, \ldots, m)$, with $C_{ij}$ as in \eqref{eq:c}. 

The contraction rate of the posterior will depend on the regularity of the true 
intensity function,  measured in H\"older sense. For $\alpha > 0$, 
let $C^\alpha[0,T]$ be the space of functions on $[0,T]$ with H\"older smoothness 
$\alpha$. (For $\lfloor\alpha\rfloor$ the greatest integer strictly smaller than 
$\alpha$, having $f \in C^\alpha[0,T]$ means that $f$ has $\lfloor\alpha\rfloor$ 
derivatives and that the highest derivative $f^{(\lfloor\alpha\rfloor)}$
is H\"older-continuous of order $\alpha - \lfloor\alpha\rfloor$.)

\begin{theorem}[Contraction rate for the spline prior]\label{theorem:spline}
Assume the true intensity function $\l_0$ belongs to $C^\alpha[0,T]$
for some $\alpha \in (0,q]$, and $M_1 \le \l_0 \le M_2$. Consider the prior $\Pi$
constructed in Section \ref{sec:prior}. For all $p >1 $ and all sufficiently large $M > 0$
we have
\[
\EE_{\l_0}\Pi_n(\lambda \in \Lambda: \|\l-  \l_0\|_2 \ge 
M\Big(\frac{n}{\log^p n}\Big)^{-\frac{\alpha}{1+2\alpha}} \given N^n ) \to 0
\]
and 
\[
\EE_{\l_0}\Pi_n(\lambda \in \Lambda: \rho(\l, \l_0) \ge M
\Big(\frac{n}{\log^p n}\Big)^{-\frac{\alpha}{1+2\alpha}} \given C^n ) \to 0
\]
as $n \to \infty$.
\end{theorem}

Note that up to a logarithmic factor, the rate of contraction in the theorem is the optimal rate\index{Rate of convergence!minimax} $n^{-\alpha/(1+2\alpha)}$ for estimating an $\alpha$-regular function. 
Moreover, the prior does not depend on $\alpha$\index{Estimator!adaptive}.
Hence the procedure automatically adapts to the smoothness of the intensity function, up to the order of the splines that are used.
This theorem deals with the case  that we have known bounds $M_1$ and $M_2$ for the intensity.
The existence lower bound $M_1>0$ is not restrictive since it can be enforced by adding a homogeneous Poisson process with known intensity to the data.

\section{Concluding remarks}
\label{sec:con}

In this paper we work specific spline-based prior for doing nonparametric Bayesian intensity smoothing for inhomogeneous Poisson processes. 
We show that the method is both practically feasible and is underpinned by theoretical performance guarantees in the form of adaptive rate-optimality results. 

Extensions of our results in several directions are possible. 
In particular, with more work it is possible to drop the assumption that we know an a-priori bound on the unknown intensity, which may be undesirable or impossible in certain situations.
An obvious extension is then to put a prior on the  bound.
Computationally this makes the procedure more demanding, but numerical investigations indicate it is still feasible.
Theoretical results can be obtained for that more general setting as well. 
Among other things this involves an extension of Lemma \ref{lemma:bounds}. 
Having a prior on the upper bound for $\lambda_0$ may deteriorate the convergence rate however.
We expect that the optimal rate will only be attained if the prior on the bound has sufficiently thin tails.

Another desirable theoretical extension would be to obtain ``local'' rate of convergence results.
Our present results deal with global norms on the intensity functions. 
It is conceivable however that convergence is faster in regions where the intensity fluctuates relatively little, and faster in others. 
More work is necessary to derive theorems that describe this phenomenon.

\bigskip

\appendix

\section{Proofs}

\subsection{Proof of Theorem \ref{theorem:con}}
\label{sec:a1}

A useful observation is that we can view the statistical problem to which 
the  theorem applies as a density estimation problem for functional data. 
Indeed, in the full observations case we observe a sample $N^{(1)}, \ldots, N^{(n)}$, 
which are independent, identically distributed random elements in the Skorohod space $D[0,T]$
of c\`adl\`ag (right-continuous functions with left-hand limits) on $[0,T]$ (see \citealt{Jacod:2003}, 
Chapter VI). Under the intensity function $\lambda$, the density $p_\lambda$ of $N^{(1)}$ relative to the 
law of a standard Poisson process indexed by $[0,T]$ is given by 
\[
p_\lambda(N) = e^{-\int_0^T(\lambda(t)-1)\,dt + \int_0^T \log(\lambda(t))\,dN_t}
\]
(e.g.\ \citealt{Jacod:2003}, Chapter III). 
Hence, the density estimation results of 
\cite{Ghosal:2000}, \cite{Ghosal:2001} apply in our case.

We want to apply Theorem 2.1 of \cite{Ghosal:2001}. This gives conditions for 
posterior contraction rates in terms of the Hellinger distance on densities and other, 
related distance measures. 
The Hellinger distance $h(p_\l, p_{\l'})$ is in our case given by 
$h^2(p_\l, p_{\l'})  = 2 (1- \EE_{\l'}\sqrt{p_\l(N)/p_{\l'}(N)})$,
where $\EE_\l$ is the expectation corresponding to the probability measure $\PP_\l$ under which 
the process $N$ is a Poisson process with intensity function $\l$. 
The other relevant distance measures are the Kullback-Leibler divergence $K(p_\l, p_{\l'}) =  - \EE_{\l'} \log(p_\l(N)/p_{\l'}(N))$ between $p_\l$ and $p_{\l'}$
and the related 
variance measure  $V(p_\l, p_{\l'}) =   \Var_{\l'} \log(p_\l(N)/p_{\l'}(N))$.
For a Poisson process $N$ with intensity $\lambda$ and a bounded, measurable function $f$, we have
\begin{align*}
\EE \int_0^T f(t)\,dN_t & = \int_0^T f(t)\l(t)\,dt,\\
\Var \int_0^T f(t)\,dN_t & = \int_0^T f^2(t)\l(t)\,dt,\\ 
\EE e^{\int_0^T f(t)\,dN_t} & = e^{-\int_0^T (1-\exp(f(t)))\l(t)\,dt}.
\end{align*}
Using these relations it is straightforward to verify that we have
\begin{align*}
h^2(p_\lambda, p_{\lambda'}) &= 2(1-e^{-\frac12 \int_0^T\big(\sqrt{\l(t)} - \sqrt{\l'(t)}\big)^2\,dt}),\\
K(p_\l, p_{\l'}) & =  
\int_0^T(\l(t) - \l'(t))\,dt + \int_0^T\lambda'(t)\log\frac{\l'(t)}{\l(t)} \,dt,\\
V(p_\l, p_{\l'}) & =   \int_0^T\lambda'(t)\log^2\frac{\l'(t)}{\l(t)} \,dt,
\end{align*}
respectively. 

The following lemma relates these statistical distances between densities to  
certain distances between intensity functions. We denote the minimum and maximum 
of two numbers $a$ and $b$ by $a \wedge b$ and $a \vee b$, respectively.

\begin{lemma}
\label{lemma:bounds}
We have the inequalities
\begin{align*}
\frac1{\sqrt 2} \Big(\|\sqrt{\l\phantom{'}} - \sqrt{\l'}\|_2 \wedge 1\Big) \le 
h(p_{\l\phantom{'}}, p_{\l'}) & \le \sqrt 2\Big(\|\sqrt{\l\phantom{'}} - \sqrt{\l'}\|_2 \wedge 1\Big),\\
K(p_{\l\phantom{'}}, p_{\l'}) & \le 3\|\sqrt{\l\phantom{'}} - \sqrt{\l'}\|_2^2 + V(p_{\l\phantom{'}}, p_{\l'}),\\
\|\sqrt{\l\phantom{'}} - \sqrt{\l'}\|^2_2 & \le \frac14\int_0^T (\l(t) \vee \l'(t))\log^2\frac{\l(t)}{\l'(t)}\,dt.
\end{align*}
\end{lemma}

\begin{proof}
The inequalities for $h$ follow from the fact that $(1/4) (x\wedge 1) \le  1-\exp(-x/2) \le x \wedge 1$ for 
$x \ge 0$. 

For the Kullback-Leibler divergence we have
\[
K(p_\l, p_{\l'}) = \int_0^T \l'(t)f(\l(t)/\l'(t))\,dt, 
\]
for $f(x) = x-1-\log x$. By Taylor's formula, $|f(x)|$ is bounded by a constant times $(\sqrt{x}-1)^2$ 
in a neighborhood of $1$. Since $|f(x)|$ is bounded by $|x|$ for $x \ge 1$ and $|x|/(\sqrt{x}-1)^2 \to 1$
as $x \to \infty$, we have in fact $|f(x)| \le 3 (\sqrt{x}-1)^2$ for all $x \in (1/e, \infty)$, say. 
For $(0, 1/e)$ we have 
$|f(x)| \le |\log x|$. It follows that 
\[
K(p_\l, p_{\l'}) \le 3\int_{\l/\l' \ge 1/e}(\sqrt{\l(t)} - \sqrt{\l'(t)})^2\,dt + \int_{\l/\l' \le 1/e} \l'(t)\Big|\log \frac{\l(t)}{\l'(t)}\Big|\,dt.
\]
The first term on the right is bounded by $3\|\sqrt\l - \sqrt{\l'}\|_2^2$. For the second term we note that 
 for $\lambda/\l' \le 1/e$,  we have $|\log (\l/\l')| \ge 1$ and hence  $|\log (\l/\l')| \le \log^2 |\l/\l'|$. 
The statement of the lemma follows.

To prove the last inequality, write $\|\sqrt{\l\phantom{'}} - \sqrt{\l'}\|^2_2$ as the sum of an integral over 
the set $\{\l' \le \l\}$ and an integral over the set $\{\l' > \l\}$ and use the fact that $1-x \le |\log x|$
for $x \in (0,1)$. 
\end{proof}

To connect assumptions \eqref{eq:pm}--\eqref{eq:en} to the corresponding assumptions 
of Theorem 2.1 of \cite{Ghosal:2001} we first note that 
since $\l_0$ is bounded away from $0$ and infinity by assumption, the same holds for any $\l \in \L$ 
that is uniformly close enough to $\l_0$. 
The lemma and the definition of $V$ therefore imply that for $\l$ uniformly close enough to $\l_0$, 
both $K(p_\l, p_{\l_0})$ and $V(p_{\l}, p_{\l_0})$ are bounded by a constant times the uniform norm 
$\|\l-\l_0\|_\infty$.
It follows  that for $n$ large enough, the Kullback-Leibler-type ball
\[
B(\eps_n) = \{\lambda \in \Lambda: K(p_\lambda, p_{\lambda_0}) \le \tilde\eps^2_n, 
V(p_\l, p_{\l_0}) \le \tilde\eps^2_n\}
\]
is larger than a multiple of the uniform ball $\{\lambda \in \Lambda : \|\l-\l_0\|_\infty \le \tilde\eps_n\}$. 
The lemma also implies that the covering number $N(\bar\eps_n, \{p_\l: \l \in \Lambda_n\}, h)$ 
is bounded by $N(\bar\eps_n/\sqrt{2}, \sqrt\Lambda_n, \|\cdot\|_2)$. 
Hence, assumptions \eqref{eq:pm}--\eqref{eq:en} imply that the conditions of 
Theorem 2.1 of \cite{Ghosal:2001} are fulfilled. This theorem states that for $M$ large enough, 
$\EE_{\l_0} \Pi_n(\l: h(p_\l, {p_{\l_0}}) \ge M \eps_n) \to 0$. 
To complete the proof, note that by the fact that $M\eps_n \le 1$ for $n$ large enough and the first inequality of 
the lemma, it holds, 
for $n$ large enough, that $\|\sqrt\l-\sqrt\l_0\|_2 \ge \sqrt{2}M\eps_n$ implies that $h(p_\l, {p_{\l_0}}) \ge M \eps_n$.

\subsection{Proof of Theorem \ref{theorem:dis}}
\label{sec:a2}

The proof is similar as the proof of Theorem \ref{theorem:con}, but this time 
we start from the observation that in the discrete-observations case, 
the data constitute a sample of  $n$ independent, identically distributed random vectors 
$C^{(1)}, \ldots, C^{(n)}$ in $\RR^m$, where
\[
C^{(i)} = (C_{i1}, \ldots, C_{im})
\]
and $C_{ij}$ is given by \eqref{eq:c}.
The coordinates $C_{ij}$ of $C^{(i)}$ are independent Poisson variables with mean $\l_j$
given by \eqref{eq:l}.

Again we  apply Theorem 2.1 of \cite{Ghosal:2001}. 
In this case the 
Hellinger distance $h_m$, Kullback Leibler divergence $K_m$ and variance measure 
$V_m$ are easily seen to be given by
\begin{align*}
h_m^2(\lambda, {\lambda'}) &= 2(1-e^{-\frac12 \sum\big(\sqrt{\l_j\phantom{'}} - \sqrt{\l_j'}\big)^2}),\\
K_m(\l, {\l'}) & =  
\sum(\l_j - \l'_j) + \sum\lambda'_j\log\frac{\l'_j}{\l_j},\\
V_m(\l, {\l'}) & =   \sum\lambda'_j\log^2\frac{\l'_j}{\l_j},
\end{align*}
respectively. These quantities satisfy the same bounds as in Lemma \ref{lemma:bounds}, but 
with the integrals replaced by the corresponding sums. 
Moreover, by expanding the square and using Cauchy-Schwarz we see that
\[
\sum \Big(\sqrt{\l_j\phantom{'}} - \sqrt{\l_j'}\Big)^2 \le \|\sqrt{\l\phantom{'}} - \sqrt{\l'}\|^2_2,
\]
and hence also 
\[
V_m(\l, \l') \le 4 \sum \frac{\l_j'}{\l_j \wedge \l_j'}
\Big(\sqrt{\l_j\phantom{'}} - \sqrt{\l_j'}\Big)^2 \le 
4 \frac{\|\l'\|_\infty}{\inf_t |\l(t)| \wedge \inf_t |\l'(t)|}\|\sqrt{\l\phantom{'}} - \sqrt{\l'}\|^2_2.
\]
Using these relations the proof can be completed exactly as in Section \ref{sec:a1}.

\subsection{Proof of Theorem \ref{theorem:spline}}

Under the prior $\Pi$, the number of knots $J$ has, by construction, 
a shifted Poisson distribution. By Stirling's approximation, this implies that for large $j$, 
\[
\PP(J > j) \lle e^{-c_1j\log j}, \qquad
 \PP(J = j) \gtrsim e^{-c_2j\log j}
\]
for some $c_1, c_2 > 0$. 
For the sequence of inner knots $\bm{k}$ constructed in the definition of the prior we have
that the mesh width $M(\bm{k}) = \max\{|k_j-k_{j-1}|\}$ and the sparsity  
$m(\bm{k}) = \min\{|k_j-k_{j-1}|\}$ satisfy 
\begin{align*}
\PP(m(\bm{k}) < j^{-2} \given J = j) = 0, \qquad
\PP(M(\bm{k}) \le 2/j \given J = j) \gtrsim e^{-j\log j}.
\end{align*}
The first of these facts follows trivially from the construction, the second one 
by bounding the probability of interest from below by the probability that 
every of the consecutive intervals of length $T/j$ contains at least one knot. 
For the B-spline coefficients we have, by independence, 
\[
\PP(\|\bm{\theta} - \bm{\theta_0}\|_\infty \le \eps \given J=j) \gtrsim \eps^j
\] 
for all $\bm{\theta_0} \in [M_1, M_2]^{j}$. 
Theorem 1 of \cite{Belitser:2012b} deals exactly with this situation. In the present setting the theorem asserts
that if $\l_0 \in C^\alpha[0,T]$ and $M_1 \le \l_0 \le M_2$, then for $J_n, \bar J_n > q$ and positive $\eps_n \ge \bar\eps_n$ 
such that $\eps_n \to 0$, $n\bar\eps_n^2 \to \infty$ and
\begin{equation}\label{eq:cc}
2 \Big(\frac{\bar\eps_n}{\|\l_0\|_{C^\alpha}}\Big)^{-1/\alpha} \le \bar J_n, \quad
\log \bar J_n \lle \log \frac 1{\bar\eps_n}, \quad
J_n \log\frac{J_n^3}{\eps_n} \lle n\eps^2_n, \quad
n\bar\eps_n^2 \le J_n \log J_n,
\end{equation}
then there exist function spaces (of splines) $\L_n$ and a constant $c > 0$ such that 
\begin{align}
\label{eq:pm2}\Pi(\lambda: \|\l-\l_0\|_\infty \le 2\bar\eps_n) & \gtrsim e^{-c\bar J_n\log\frac1{\bar\eps_n}},\\
\label{eq:rm2}\Pi(\lambda \not \in \L_n) & \lle e^{-c_1n\bar\eps^2_n},\\
\label{eq:en2}\log N(\eps_n, \L_n, \|\cdot\|_2) & \lle n\eps^2_n.
\end{align}

Now observe that  the first two inequalities in \eqref{eq:cc} hold for 
\[
\bar\eps_n = n^{-\frac{\alpha}{1+2\alpha}}\log^p n, \quad
\bar J_n = Kn^{\frac{1}{1+2\alpha}}\log^{q} n,
\]
provided $K$ is large enough and $q \ge -p/\alpha$. The third and fourth inequalities then hold for 
\[
 J_n = Ln^{\frac{1}{1+2\alpha}}\log^{r} n, \quad
\eps_n = n^{-\frac{\alpha}{1+2\alpha}}\log^s n
\] 
if $L$ is large enough and $2p \le r+1 \le 2s$. 
To complete the proof we have to link \eqref{eq:pm2}--\eqref{eq:en2} to 
the conditions \eqref{eq:pm}--\eqref{eq:en} of Theorems \ref{theorem:con}
and \ref{theorem:dis}. 
Note that since \eqref{eq:rm} should hold for all $c_3 > 0$, we need to have 
\[
\bar J_n \log\frac1{\bar\eps_n} \ll n\bar\eps^2_n.
\]
For our choices of $\bar J_n$ and $\bar\eps_n$ this holds if $2p > q+1$. 
This amounts to choosing $p> \alpha/(1+2\alpha)$.
Then if we define 
\[
\tilde\eps_n = \sqrt{\frac{\bar J_n}{n} \log\frac1{\bar\eps_n}}
\]
the right-hand side of \eqref{eq:pm2} equals $\exp(-2n\tilde\eps^2_n)$. 
Moreover, it holds that $\tilde\eps_n \sim n^{-\alpha/(1+2\alpha)}(\log n)^{(q+1)/2}$, 
so if we  make sure that $p > (q+1)/2$, the desired inequality \eqref{eq:pm} holds. 
The considerations above imply
that \eqref{eq:rm} then holds as well, for any $c_3 \ge 1$. 
Recall that we found that the entropy condition holds for $\eps_n \sim n^{-\alpha/(1+2\alpha)}(\log n)^{s}$, 
provided $s > p$. This means that we should choose $p,q, r$ and $s$ above such that 
\[
p > \frac{\alpha}{1+2\alpha}, \quad
r = 2p-1, \quad
s > p, \quad
q = -\frac1{1+2\alpha}.
\]
Since the intensities in 
$\Lambda_n$ are uniformly bounded by a common constant (see the proof of Theorem 1 
of \citealt{Belitser:2012b}), \eqref{eq:en2} implies that \eqref{eq:en} is fulfilled.

\section*{Acknowledgement}

Research supported by the Netherlands Organization for Scientific Research NWO.

\bibliographystyle{biometrika}
\bibliography{ref}

\end{document}